\documentclass[utf8,11pt]{article}
\usepackage[utf8]{inputenc}
\usepackage{fullpage}
\usepackage{authblk}

\usepackage{amsthm,amsmath,amssymb}

\newcommand{\qedclaim}{\hfill $\diamond$ \medskip}

\usepackage[colorlinks=true,citecolor=black,linkcolor=black,urlcolor=blue]{hyperref}

\theoremstyle{plain}
\newtheorem{theorem}{Theorem}
\newtheorem{lemma}[theorem]{Lemma}

\theoremstyle{definition}

\newtheorem{conjecture}[theorem]{Conjecture}

\newtheorem{problem}[theorem]{Problem}

\begin{document}

\title{A proof of the Erd\H{o}s-Sands-Sauer-Woodrow conjecture}

\author{Nicolas Bousquet$^a$, William Lochet$^{b,c}$, and St\'{e}phan Thomass\'{e}$^b$\\~\\
\small $^a$Laboratoire G-SCOP\\ \small 46, avenue Félix Viallet
\\ \small 38031 Grenoble Cedex 1 - France\\~\\
\small $^b$Laboratoire d'Informatique du Parall\'elisme \\ \small   UMR 5668 ENS Lyon - CNRS - UCBL - INRIA
\\ \small Universit\'e de Lyon, France\\~\\
\small $^c$ Universit\'e C\^ote d'Azur, CNRS, I3S, UMR 7271, Sophia Antipolis, France}

\date{}

\maketitle
\begin{abstract}
A very nice result of B\'ar\'any and Lehel asserts that every finite subset $X$ or $\mathbb R^d$ can be covered 
by $f(d)$ $X$-boxes (i.e. each box has two antipodal points in $X$). As shown by Gy\'arf\'as and 
P\'alv\H{o}lgyi this result would follow from the following conjecture : If a tournament admits 
a partition of its arc set into $k$ quasi orders, then its domination number is bounded in terms of $k$.
This question is in turn implied by the Erd\H{o}s-Sands-Sauer-Woodrow conjecture : If the arcs of a tournament $T$ are colored with $k$ colors, there is a set $X$ of at most $g(k)$ vertices such that for every vertex $v$ of $T$, there is a monochromatic path from $X$ to $v$. We give a short proof of this statement. We moreover show that the general Sands-Sauer-Woodrow conjecture (which as a special case implies the stable marriage theorem) is valid for directed graphs with bounded stability number. This conjecture remains however open.
\end{abstract}

\section{Definition}

We interpret a quasi order on a set $S$ as a digraph, where the vertices are the elements 
of $S$ and the arcs are the couples $(x,y)$ such that $x\leq y$. 
In a digraph $D= (V,A)$, we define for every $x\in V$ the \textit{closed in-neighbourhood} 
$N^-(x)$ (resp the \textit{closed out-neighbourhood} $N^+(x)$)
as $\{x\}$ union the set of vertices $y$ such that $(y,x) \in A$ (resp $(x,y) \in A$).
By extension, $N^+(S)=\cup_{x \in S}N^+(x)$ when $S$ is a subset of vertices.

A set of vertices $S$ such that $N^+(S)=V$ is said to be \textit{dominating}. The \textit{domination number} $\gamma(D)$ of a digraph $D$
is the size of the smallest dominating set. A \textit{complete multidigraph} is a directed graph
in which multiple arcs and circuits of length two are allowed and such that there always exists
an arc between two distinct vertices (when the digraph is simple, we speak of a \textit{tournament}).
Let $T$ be a complete multidigraph whose arcs are the union of $k$ quasi orders $P_1, \dots, P_k$, 
we define $N^-_i(x)$ (resp $N^+_i(x)$) as the closed in-neighbourhood of the digraph induced by $P_i$.
For a function $f:T \rightarrow \mathbb{R}$, we note $f(T) = \sum_{x \in T}f(x)$.
We say that $f:S\rightarrow [0,1]$ is a \textit{probability distribution} on $S$ if $f(S) = 1$.
When $T\subseteq S$ and $f$ is a probability distribution on $T$ we often implicitely 
use $f$ as a probability distribution on $S$ by setting it to 0 on $S\setminus T$.

In \cite{SandSauer} Sands, Sauer and Woodrow asked the following question, also due to Erd\H{o}s:

\begin{problem}
For each $n$, is there a (least) positive integer $f(n)$ so that
every finite tournament whose edges are coloured with $n$ colours contains a
set $S$ of $f(n)$ vertices with the property that for every vertex $u$ not in $S$ there
is a monochromatic path from $u$ to a vertex of $S$?
\end{problem}

The transitive closure of each color class is a quasi-order (transitive digraph), hence the 
Erd\H{o}s-Sands-Sauer-Woodrow can be restated as:

\begin{conjecture}\label{cj:main}
For every $k$, there exists an integer $f(k)$ such that if $T$ is a complete multidigraph whose arcs are the union of $k$ quasi-orders, then $\gamma(T) \leq f(k)$.
\end{conjecture}

In \cite{Gyarfas}, Gyárfás and Pálvölgyi asked a weaker conjecture where $T$ is a tournament instead of a complete multidigraph, which means that the (quasi)
orders are disjoint. This renewed attention on the Erd\H{o}s-Sands-Sauer-Woodrow problem
since they showed that it directly implies that every finite subset $X$ or $\mathbb R^d$ can be covered 
by $f(d)$ $X$-boxes (see \cite{Alon06}, \cite{Bara}, \cite{Pa}).

Despite some attention (see \cite{G}, \cite{GS}, \cite{Melcher}, \cite{Min}, \cite{PRRSL} for example)
the case $k=3$ was still open in both conjectures. The main goal of this paper is to prove Conjecture \ref{cj:main}. In their seminal paper, Sands, Sauer and Woodrow also proposed a more general open problem:

\begin{conjecture}\label{cj:real}
For every $k$, there exists an integer $f(k)$ such that if $D$ is a multidigraph whose arcs are the union of $k$ quasi-orders, then $D$ has a dominating set which is the union of $f(k)$ stables sets.
\end{conjecture}

This statement (investigated in \cite{AharoniBerger}, \cite{AharoniFleiner}, \cite{Fleiner}, \cite{FleinerJanko}) is still open. We provide a proof of it when the maximal stable 
set has some fixed value.

\section{Proof}

The next result is a direct application of Farkas' lemma (or follows from the 
existence of mixed strategies in matrix games).

\begin{lemma}\label{lem:parti}
Let $T$ be a complete multidigraph whose arc set is the union of $k$ quasi-orders. There exists a 
probability distribution $w$ on $V(T)$ and a partition of $V(T)$ into sets 
$T_1, T_2, \dots, T_k$ such that for every $i$ and $x \in T_i$, $w(N^-_i(x)) \geq 1/2k$. 
\end{lemma}

\begin{proof}
By Theorem 1 in \cite{Fisher} there exists a weight function $w: V(T) \rightarrow [0,1]$ such that $w(N^-(x)) \geq 1/2$ 
for all $x \in T$. For every $i$ in $[k]$, let $T'_i$ be the subset of vertices such that $w(N^-_i(x)) \geq 1/2k$. The sets $T'_i$ cover the vertices, so we can extract a partition with the required properties. 
\end{proof}

Let $P$ be a quasi-order on $S$. We say that $A \subseteq P$ is $\epsilon$-$dense$ in $P$ 
if there exists a probability distribution $w$ of $P$ such that $w(N^-(x)) \geq \epsilon$ for every element $x$ of $A$. 

\begin{lemma}\label{lem:domi}
Let $\epsilon$ be a real in $[0,1]$. There exists an integer $g(\epsilon)$ such that for every quasi-order $P$ on a set $A$ and two subsets $C \subset B$ of $A$ such that $B$ is $\epsilon$-dense in $P$ and $C$ is $\epsilon$-dense in $B$, there exists a set of $g(\epsilon)$ elements in $A$ dominating $C$. 
\end{lemma}

\begin{proof}
Let $w_A:A \rightarrow [0,1]$ and $w_B: B\rightarrow [0,1]$ be the probability distributions such that
$w_A(N^-(x)) \geq \epsilon$ for every $x \in B$ and $w_B(N^-(x)) \geq \epsilon$ for every $x \in C$. 

Let $g(\epsilon) = \lfloor \frac{ln(\epsilon)}{ln(1-\epsilon)} \rfloor + 1$ and pick independently at random according to the distribution 
$w_A$ a (multi)set $S$ of $g(\epsilon)$ elements of $A$.
For every vertex $x \in B$, $P(x \in N^+(S)) \geq 1 - (1-\epsilon)^{g(\epsilon)} > 1 - \epsilon$.
Thus, by linearity of $w_B$ and the expectation, $\mathbb{E}(w_B(N^+(S))) > 1-\epsilon$. Therefore, there exists a choice of $S$ 
such that $w_B(N^+(S)) > 1-\epsilon$. Since $w_B(N^-(y)) \geq \epsilon$ for
every $y \in C$, the set $N^-(y)$ intersects $N^+(S)$. In particular, by transitivity, $S$ dominates 
$y$.

\end{proof}
We are now ready to prove our main result: 

\begin{theorem}\label{th:main}
For every $k$, if $T$ is a complete multidigraph whose arcs are the union of $k$ 
quasi-orders, then $\gamma(T) =O(ln(2k) \cdot k^{k+2})$.
\end{theorem}

\begin{proof}
Consider $P_1 = T_1, T_2, \dots, T_k$ together with $w$ the partition given by Lemma \ref{lem:parti} applied to $T$.
Each of the $T_i$ is a complete multidigraph which is the union of $k$ quasi orders, this means we can apply 
Lemma \ref{lem:parti} and obtain $T_{i,1}$, $T_{i,2}$ $\dots$, $T_{i,k}$ together with a probability distribution $w_i$  on $T_i$
such that $w_i(N^-_j(x)) \geq 1/2k$ for every $x \in T_{i,j}$. 
By repeating this process $k$ times, we obtain a sequence of $k+1$ partitions $P_1, \dots, P_{k+1}$ with  
$P_i = \cup_{j_1,j_2, \dots, j_i \leq k} T_{j_1,j_2,\dots,j_i}$ such that for every $l\leq k+1$ and each $j_1, \dots, j_l$ in $[k]^l$,
$T_{j_1,j_2,\dots,j_l}$ is a subset of $T_{j_1,j_2,\dots,j_{l-1}}$ and the probability distribution $w_{j_1 \dots, j_{l-1}}$
is such that $w_{_1 \dots, j_{l-1}}(N^-_{j_l}(x)) \geq 1/2k$ for every $x$ in $T_{j_1, \dots, j_l}$.

Fix $j_1, \dots, j_{k+1}$, in $[k]^{k+1}$, by the pigeonhole principle there exists two indices, $i<l$ such that $j_i = j_l$, 
then by applying lemma \ref{lem:domi} where $T_{j_1, \dots, j_{i-1}}$, $T_{j_1, \dots, j_{i}}$ and $T_{j_1, \dots, j_{l}}$
play the roles of, respectively, $A$, $B$ and $C$ there exists a set of size $g(1/2k)$ that dominates $T_{j_1, \dots, j_{l}}$ and thus
$T_{j_1, \dots, j_{k+1}}$. This means that $\gamma(T) \leq k^{k+1}\cdot g(1/2k)$. 
Moreover since $g(1/2k) \leq  ln(2k)*(2k-1/2 + o(1))$, we have $\gamma(T) =O(ln(2k) \cdot k^{k+2})$.

\end{proof}

The main ingredient in our proof is that the fractional domination 
of complete multidigraphs is bounded. Since this also holds when the stability number
is bounded (in fact the fractional domination of vertices by stable sets is
at most 2), the general Sands-Sauer-Woodrow conjecture holds for classes
of digraphs with bounded stability.


\begin{thebibliography}{XX}

\bibitem{AharoniBerger}
R. Aharoni, E. Berger and I. Gorelik,
\newblock Kernels in Weighted Digraphs.
\newblock \textit{Order}, \textbf{31}, (2014), 35-43

\bibitem{AharoniFleiner}
 R. Aharoni, T. Fleiner, 
\newblock On a lemma of Scarf.
\newblock \textit{ Journal of Combinatorial Theory, Series B}, \textbf{87}, (2003), 72-80. 

\bibitem{Alon06}N. Alon, G. Brightwell, H. A. Kierstead, A. V. Kostochka and P. Winkler,
\newblock Dominating sets in k-majority tournaments.
\newblock \textit{ Journal of Combinatorial Theory, Series B}, \textbf{96},
(2006), 374-387.

\bibitem{Bara}I. B\'ar\'any and J. Lehel,
\newblock Covering with Euclidean boxes.
\newblock \textit{European Journal of Combinatorics},
\textbf{8}, (1987), 113-119.

\bibitem{Fisher}D. Fisher and J. Ryan, 
\newblock Probabilities within optimal strategies for tournament games.
\newblock \textit{Discrete Applied Math.}, \textbf{56}, (1995), 87–91.

\bibitem{Fleiner}T. Fleiner,
\newblock A Fixed-Point Approach to Stable Matchings and Some Applications.
\newblock \textit{Mathematics of Operations Research}, \textbf{28}, (2003), 103-126.

\bibitem{FleinerJanko}
T. Fleiner and Z. Jank\'o,
\newblock On Weighted Kernels of Two Posets,
\newblock \textit{Order}, \textbf{33}, (2016), 51-65.

\bibitem{G}H. Galeana-Sanchez and R. Rojas-Monroy,
\newblock Monochromatic paths and at most 2-colored arc sets in edge-coloured tournaments.
\newblock\textit{Graphs Combin.}, \textbf{21}, (2005), 307–317.

\bibitem{GS}A. Georgakopoulos and P. Spr\"ussel,
\newblock On 3-colored tournaments.
\newblock \textit{preprint}, (2009).

\bibitem{Gyarfas} A. Gyárfás and D. Pálvölgyi,
\newblock Domination in transitive colorings of tournaments.
\newblock \textit{ Journal of Combinatorial Theory, Series B}, \textbf{107}, (2014), 1-11.

\bibitem{Melcher}M. Melcher and K.B. Reid,
\newblock Monochromatic sinks in nearly transitive arc-colored tournaments.
\newblock \textit{Discrete Math.}, \textbf{310}, (2010), 2697–2704.

\bibitem{Min}S. Minggang, 
\newblock On monochromatic paths in m-coloured tournaments.
\newblock \textit{Journal of Combinatorial Theory}, Series B, \textbf{45}, (1988), 108-111.

\bibitem{Pa}J. Pach,
\newblock A remark on transversal numbers,
\newblock in: R.L. Graham, J. Nesetril (Eds.), The Mathematics
of Paul Erdős, vol. II, Springer, (1997), 310-317.

\bibitem{PRRSL}L. Pastrana-Ramirez and M. del Rocio Sanchez-Lopez,
\newblock Kernels by monochromatic directed paths in 3-colored tournaments and quasi-tournaments.
\newblock \textit{Int. J. contemp. Math. sciences}, \textbf{5}, (2010), 1689-1704.

\bibitem{SandSauer}B. Sands, N. Sauer and R. Woodrow,	
\newblock On monochromatic paths in edge-coloured digraphs.
\newblock \textit{Journal of Combinatorial Theory, Series B}, \textbf{33}, (1982), 271-275.



\end{thebibliography}
\end{document}